\documentclass{article}
\usepackage[utf8]{inputenc}
\usepackage[height=8.2in,width=6in,letterpaper]{geometry}
\usepackage{natbib}
\usepackage[colorlinks=true,citecolor=blue,breaklinks]{hyperref}
\usepackage[hyphenbreaks]{breakurl} 

\makeatletter
\newlength\aftertitskip     \newlength\beforetitskip
\newlength\interauthorskip  \newlength\aftermaketitskip

\def\maketitle{\par
 \begingroup
   \def\thefootnote{\fnsymbol{footnote}}
   \def\@makefnmark{\hbox to 4pt{$^{\@thefnmark}$\hss}}
   \@maketitle \@thanks
 \endgroup
\setcounter{footnote}{0}
 \let\maketitle\relax \let\@maketitle\relax
 \gdef\@thanks{}\gdef\@author{}\gdef\@title{}\let\thanks\relax}

\def\@startauthor{\noindent \normalsize\bf}
\def\@endauthor{}
\def\@starteditor{\noindent \small {\bf Editor:~}}
\def\@endeditor{\normalsize}
\def\@maketitle{\vbox{\hsize\textwidth
 \linewidth\hsize \vskip \beforetitskip
 {\begin{center} \LARGE\@title \par \end{center}} \vskip \aftertitskip
 {\def\and{\unskip\enspace{\rm and}\enspace}%
  \def\addr{\small}%
  \def\email{\hfill\small\sf}%
  \def\name{\normalsize\bf}%
  \def\AND{\@endauthor\rm\hss \vskip \interauthorskip \@startauthor}
  \@startauthor \@author \@endauthor}
}}

\makeatother

\pdfoutput=1                    

\usepackage{floatrow} 
\newfloatcommand{capbtabbox}{table}[][\FBwidth]

\usepackage{graphicx}
\usepackage{amsmath,amsthm,amssymb,bm} 
\usepackage{amsfonts}
\usepackage{mathrsfs}
\usepackage{subfigure}
\usepackage{xspace}
\usepackage{array}
\usepackage{enumerate}
\usepackage{algorithm}
\usepackage{algorithmic}
\usepackage{appendix}
\usepackage{wrapfig}
\numberwithin{equation}{section}

\newtheorem{theorem}{Theorem}

\newtheorem{prop}[theorem]{Proposition}
\newtheorem{corr}[theorem]{Corollary}
\newtheorem{conj}[theorem]{Conjecture}
\theoremstyle{definition}

\newcommand{\nlsum}{\sum\nolimits}
\newcommand{\nlprod}{\prod\nolimits}
\newcommand{\R}{\mathbb{R}}
\newcommand{\vx}{\bm{x}}

\newcommand{\pfrac}[2]{\left(\tfrac{#1}{#2}\right)}
\newcommand{\half}{\tfrac{1}{2}}
\DeclareMathOperator*{\tr}{tr}
\DeclareMathOperator*{\Diag}{Diag}

\begin{document}
\title{On inequalities for normalized Schur functions\thanks{To appear in \emph{European J. Combinatorics}, DOI: 10.1016/j.ejc.2015.07.005}}
\author{\name Suvrit Sra \email suvrit@mit.edu\\ 
\addr Laboratory for Information and Decision Systems\\ Massachusetts Institute of Technology\\ Cambridge, MA, 02139, United States}

\maketitle

\vskip0.4cm
\hrule
\vskip0.4cm

\begin{abstract}
  We prove a conjecture of Cuttler et al.~[2011] [A. Cuttler, C. Greene, and M. Skandera; \emph{Inequalities for symmetric means}. European J. Combinatorics, 32(2011), 745--761] on the monotonicity of \emph{normalized Schur functions} under the usual (dominance) partial-order on partitions. We believe that our proof technique may be helpful in obtaining similar inequalities for other symmetric functions.
\end{abstract}

We prove a conjecture of \citet{cuttler} on the monotonicity of normalized Schur functions under the majorization (dominance) partial-order on integer partitions.

Schur functions are one of the most important bases for the algebra of symmetric functions. Let $\vx  = (x_1,\ldots,x_n)$ be a tuple of $n$ real variables. Schur functions of $\vx$ are indexed by integer partitions $\lambda=(\lambda_1,\ldots,\lambda_n)$, where $\lambda_1 \ge \cdots \ge \lambda_n$, and can be written as the following ratio of determinants~\citep[pg.~49]{schur.phd}, \citep[(3.1)]{mac95}:
\begin{equation}
  \label{eq:1}
  s_\lambda(\vx) = s_\lambda(x_1,\ldots,x_n) := \frac{\det([x_i^{\lambda_j+n-j}]_{i,j=1}^n)}{\det([x_i^{n-j}]_{i,j=1}^n)}.
\end{equation}
To each Schur function $s_\lambda(\vx)$ we can associate the \emph{normalized Schur function}
\begin{equation}
  \label{eq:2}
  S_\lambda(\vx) \equiv S_\lambda(x_1,\ldots,x_n) := \frac{s_\lambda(x_1,\ldots,x_n)}{s_\lambda(1,\ldots,1)} = \frac{s_\lambda(\vx)}{s_\lambda(1^n)}.
\end{equation}
Let $\lambda$, $\mu \in \mathbb{R}^n$ be decreasingly ordered. We say $\lambda$ is \emph{majorized} by $\mu$, denoted $\lambda \prec \mu$, if
\begin{equation}
  \label{eq:4}
  \begin{split}
    \nlsum_{i=1}^k\lambda_i &= \nlsum_{i=1}^k\mu_i\quad\text{for}\ 1\le i \le n-1,\quad\text{and}\quad \nlsum_{i=1}^n \lambda_i = \nlsum_{i=1}^n \mu_i.
  \end{split}
\end{equation}
\citet{cuttler} studied normalized Schur functions~\eqref{eq:2} among other symmetric functions, and derived inequalities for them under the partial-order~\eqref{eq:4}. They also conjectured related inequalities, of which perhaps Conjecture~\ref{conj:one} is the most important. 

\begin{conj}[\protect{\citep{cuttler}}]
  \label{conj:one}
  Let $\lambda$ and $\mu$ be partitions; and let $\vx \ge 0$. Then,
  \begin{equation*}
    S_\lambda(\vx) \le S_\mu(\vx),\qquad\text{if and only if}\quad \lambda \prec \mu.
  \end{equation*}
\end{conj}
\noindent\citet{cuttler} established necessity (i.e., $S_\lambda \le S_\mu$ only if $\lambda \prec \mu$), but sufficiency was left open. We prove sufficiency in this paper.

\begin{theorem}
  \label{thm:main}
  Let $\lambda$ and $\mu$ be partitions such that $\lambda \prec \mu$, and let $\vx \ge 0$. Then,
  \begin{equation*}
    S_\lambda(\vx) \le S_\mu(\vx).
  \end{equation*}
\end{theorem}

Our proof technique differs completely from~\citep{cuttler}: instead of taking a direct algebraic approach, we invoke  a well-known integral from random matrix theory. We believe that our approach might extend to yield inequalities for other symmetric polynomials such as Jack polynomials~\citep{jack70} or even Hall-Littlewood and Macdonald polynomials~\citep{mac95}.

\section{Majorization inequality for Schur polynomials}
\label{sec:proof}
Our main idea is to represent normalized Schur polynomials~\eqref{eq:2} using an integral compatible with the partial-order `$\prec$'. One such integral is the \emph{Harish-Chandra-Itzykson-Zuber (HCIZ)} integral~\citep{harish,zuber}:
\begin{equation}
  \label{eq:5}
  I(A,B) := \int_{U(n)} e^{\tr(U^*AUB)}dU = c_n
  \frac{\det([e^{a_ib_j}]_{i,j=1}^n)}{\Delta(\bm{a})\Delta(\bm{b})},
\end{equation}
where $dU$ is the Haar probability measure on the unitary group $U(n)$; $\bm a$ and $\bm b$ are vectors of eigenvalues of the Hermitian matrices $A$ and $B$; 
$\Delta$ is the \emph{Vandermonde determinant} $\Delta(\bm{a}) := \nlprod_{1 \le i < j \le n} (a_j - a_i)$; and $c_n$ is the constant 
\begin{equation}
  \label{eq:16}
  c_n = \Bigl(\nlprod_{i=1}^{n-1}i! \Bigr) = \Delta([1,\ldots,n]) = \nlprod_{1 \le i < j \le n}(j-i).
\end{equation}
The following observation~\citep{harish} is of central importance to us.
\begin{prop}
  \label{prop:gorin}
  Let $A$ be a Hermitian matrix, $\lambda$ an integer partition, and $B$ the diagonal matrix $\Diag([\lambda_j + n - j]_{j=1}^n)$. Then,
  \begin{equation}
    \label{eq:6}
    \frac{s_\lambda(e^{a_1},\ldots,e^{a_n})}{s_\lambda(1,\ldots,1)} = \frac{1}{E(A)}I(A,B),
  \end{equation}
  where the product $E(A)$ is given by
  \begin{equation}
    \label{eq:7}
    E(A) = \nlprod_{1 \le i < j \le n}\frac{e^{a_i}-e^{a_j}}{a_i-a_j}.
  \end{equation}
\end{prop}
\begin{proof}
  Recall from Weyl's dimension formula that
  \begin{equation}
    \label{eq:15}
    s_\lambda(1,\ldots,1) = \nlprod_{1\le i < j \le n} \frac{(\lambda_i-i)-(\lambda_j-j)}{j-i}.
  \end{equation}
  Now use identity~\eqref{eq:15}, definition~\eqref{eq:16}, and the ratio~(\ref{eq:1}) in~\eqref{eq:5}, to obtain~\eqref{eq:6}.
\end{proof}

Assume without loss of generality that for each $i$, $x_i > 0$  (for $x_i=0$, apply the usual continuity argument). Then, there exist reals $a_1,\ldots,a_n$ such that $e^{a_i} = x_i$, whereby
\begin{equation}
  \label{eq:8}
  S_\lambda(x_1,\ldots,x_n) = \frac{s_\lambda(e^{\log x_1},\ldots,e^{\log x_n})}{s_\lambda(1,\ldots,1)} = \frac{I(\log X,B(\lambda))}{E(\log X)},
\end{equation}
where $X=\Diag([x_i]_{i=1}^n)$; we write $B(\lambda)$ to explicitly indicate $B$'s dependence on $\lambda$ as in Prop.~\ref{prop:gorin}. Since $E(\log X)>0$, to prove Theorem~\ref{thm:main}, it suffices to prove Theorem~\ref{thm:key} instead.

\begin{theorem}
  \label{thm:key}
  Let $X$ be an arbitrary Hermitian matrix. Define the map $F: \R^n \to \R$ by
  \begin{equation*}
    F(\lambda) := I(X, \Diag(\lambda)),\qquad\lambda \in \R^n.
  \end{equation*}
  Then, $F$ is Schur-convex, i.e., if $\lambda, \mu \in \R^n$ such that $\lambda \prec \mu$, then $F(\lambda) \le F(\mu)$.
\end{theorem}
\begin{proof}
  We know from~\citep[Proposition~C.2, pg.~97]{marOlk} that a convex and symmetric function is Schur-convex. From the HCIZ integral~(\ref{eq:5}) symmetry of $F$ is apparent; to establish its convexity it suffices to demonstrate midpoint convexity:
  \begin{equation}
    \label{eq:17}
    F\pfrac{\lambda+\mu}{2} \le \half F(\lambda) + \half F(\mu)\qquad\text{for}\quad \lambda,\mu\in\R^n.
  \end{equation}
  The elementary manipulations below show that inequality~\eqref{eq:17} holds.
  \begin{align*}
    F\pfrac{\lambda+\mu}{2} &=\quad\int_{U(n)}\exp\bigl(\tr\bigl[U^*XU \Diag\bigl(\tfrac{\lambda + \mu}{2}\bigr)\bigr]\bigr)dU\\
    &=\quad\int_{U(n)}\exp\bigl(\tr\bigl[\half U^*XU\Diag(\lambda) + \half U^*XU\Diag(\mu)\bigr]\bigr)dU\\
    &=\quad\int_{U(n)}\sqrt{\exp\bigl(\tr[U^*XU\Diag(\lambda)]\bigr) \cdot \exp\bigl(\tr[U^*XU\Diag(\mu)]\bigr)}dU\\
    &\le\quad\int_{U(n)}\left(\half\exp\bigl(\tr[U^*XU\Diag(\lambda)]\bigr) +
      \half\exp\bigl(\tr[U^*XU\Diag(\mu)]\bigr)\right)dU\\
    &=\quad\half F(\lambda) + \half F(\mu),
  \end{align*}
  where the inequality follows from the arithmetic-mean geometric-mean inequality. 
\end{proof}
\begin{corr}
  Conjecture~\ref{conj:one} is true.
\end{corr}


\vspace*{-4pt}
\subsection*{Acknowledgments}
I am grateful to a referee for uncovering an egregious error in my initial attempt at Theorem~\ref{thm:key}; thanks also to the same or different referee for the valuable feedback and encouragement. I thank Jonathan Novak (MIT) for his help with HCIZ references.

\bibliographystyle{abbrvnat}

\end{document}